\documentclass[a4paper,12pt]{article}
\providecommand{\MR}{\relax\ifhmode\unskip\space\fi MR }

\usepackage{setspace}
\usepackage{latexsym} 
\usepackage[english]{babel}
\usepackage{amsfonts,amsmath,amssymb,amsbsy,amsthm,enumerate}
\usepackage[mathscr]{euscript}
\usepackage{pstricks}
\usepackage{multirow}
\usepackage{verbatim}
\usepackage{color}
\usepackage{graphicx}
\usepackage{float}

\usepackage{hyperref}

\usepackage{tikz}
\usetikzlibrary{calc}

\usepackage{thm-restate}

\theoremstyle{plain}
\newtheorem{theorem}{Theorem}
\newtheorem*{theorem-nn}{Theorem}

\newtheorem{corollary}[theorem]{Corollary}

\newtheorem{claim}{Claim}[theorem]
\newtheorem{lemma}[theorem]{Lemma}
\newtheorem{proposition}[theorem]{Proposition}

\usepackage{xcolor}

\author{Fábio Botler$^1$ \and
        Wanderson Lomenha$^1$ \and
        João Pedro de Souza$^{1,2}$}
        
\date{Program System Engineering and Computer Science, COPPE\\
$^1$Universidade Federal do Rio de Janeiro \\
$^2$Colégio Pedro II}

\title{On nonrepetitive colorings of paths and cycles\thanks{
Botler is supported by CNPq (304315/2022-2),
by FAPERJ (211.305/2019 and 201.334/2022) and by UFRJ (ALV 23.733);
Lomenha is supported by FAPERJ (201.575/2023).
This study was financed in part by the Coordena\c{c}\~ao de Aperfei\c{c}oamento de Pessoal de N\'ivel Superior, Brasil (CAPES), Finance Code~001.
FAPERJ is the Rio de Janeiro Research Foundation. 
CNPq is the National Council for Scientific and Technological
Development of Conselho Nacional de Desenvolvimento Cient\'{i}fico e
Tecnol\'{o}gico do Brasil.}}

\begin{document}
\maketitle

\begin{abstract}
    We say that a sequence \(a_1 \cdots a_{2t}\)
    of integers is \emph{repetitive}
    if \(a_i = a_{i+t}\) for every \(i\in\{1,\ldots,t\}\).
    A \emph{walk} in a graph \(G\)
    is a sequence \(v_1 \cdots v_r\)
    of vertices of \(G\)
    in which \(v_iv_{i+1}\in E(G)\) for every \(i\in\{1,\ldots,r-1\}\).
    Given a \(k\)-coloring \(c\colon V(G)\to\{1,\ldots,k\}\) of \(V(G)\),
    we say that \(c\) is \emph{walk-nonrepetitive} (resp. \emph{stroll-nonrepetitive})
    if for every \(t\in\mathbb{N}\) and every walk \(v_1\cdots v_{2t}\)
    the sequence \(c(v_1) \cdots c(v_{2t})\) is not repetitive
    unless \(v_i = v_{i+t}\) for every \(i\in\{1,\ldots,t\}\)
    (resp. unless \(v_i = v_{i+t}\) for some \(i\in\{1,\ldots,t\}\)).
    The \emph{walk} (resp. \emph{stroll}) \emph{chromatic number} \(\sigma(G)\) (resp. \(\rho(G)\)) of \(G\)
    is the minimum \(k\) for which \(G\) has a walk-nonrepetitive (resp. stroll-nonrepetitive) 
    \(k\)-coloring.
    Let \(C_n\) and \(P_n\) denote, respectively, the cycle and the path with \(n\) vertices.
    In this paper we present three results that answer questions posed by Barát and Wood in 2008:
    (i) \(\sigma(C_n) = 4\) whenever \(n\geq 4\) and \(n \notin\{5,7\}\);
    (ii) \(\rho(P_n) = 3\) if \(3\leq n\leq 21\) and \(\rho(P_n) = 4\) otherwise;
    and (iii) \(\rho(C_n) = 4\), whenever \(n \notin\{3,4,6,8\}\), and \(\rho(C_n) = 3\) otherwise.
    In particular, (ii) improves bounds on \(n\) obtained by Tao in 2023. 
\end{abstract}

\section{Introduction}

In this paper, we consider only finite and undirected graphs without loops or multiple edges. 
Let $G$ be a graph on $n$ vertices. 
A \emph{walk} in \(G\) is a sequence \(v_1 \cdots v_t\) of vertices of \(G\) 
for which \(v_iv_{i+1}\in E(G)\)
for every \(i \in \{1,\ldots, t-1\}\);
and a \emph{path} in $G$ is a walk \(v_1\cdots v_t\) in $G$ for which $v_i \neq v_j$ for all distinct $i,j \in \{1,\ldots, t\}$.
Given a positive integer \(n\),
we denote by  $P_n$ (resp. \(C_n\))  the path (resp. cycle) with \(n\) vertices.

Let \(S = a_1\cdots a_{2t}\) be a sequence of integers.
We say that \(S\) is \emph{repetitive} if $a_i=a_{i+t}$ for all $i\in \{1,\ldots, t\}$;
otherwise we say that \(S\) is \emph{nonrepetitive}.
A \emph{subsequence} \(S'\) of \(S\) is any sequence of consecutive elements of \(S\).
In 1906, Thue proved that there are arbitrarily large sequences on three symbols
without repetitive subsequences~\cite{thue1906unendliche}.

Given a graph \(G\) and a positive integer \(k\), a \emph{$k$-coloring} of the vertices of $G$ is a function \(c\colon V(G)\to \{1,\ldots,k\}\).
Given a \(k\)-coloring  \(c\) of the vertices of a graph \(G\)
we say that a walk \(w = v_1\cdots v_{2t}\)
is \emph{repetitive} (with respect to \(c\)) if \(c(v_1)\cdots c(v_{2t})\) is a repetitive sequence;
and we say that \(c\) is a \emph{path-nonrepetitive coloring} if no path
of \(G\) is repetitive.
The \emph{path-nonrepetitive chromatic number} of \(G\), denoted by $\pi(G)$,
is then defined to be the minimum integer $k$ for which $G$ admits a path-nonrepetitive $k$-coloring. 
Naturally, every path-nonrepetitive coloring is a \emph{proper} coloring,
i.e., \(c(u) \neq c(v)\) for every \(uv\in E(G)\),
and hence \(\chi(G)\leq \pi(G)\),
where \(\chi(G)\) is the well-known \emph{chromatic number} of \(G\).

The concept of path-nonrepetitive colorings was introduced in 2002 by Alon, Grytczuk, Ha{\l}uszczak, and Riordan~\cite{alon2002nonrepetitive},
who proved that for every graph \(G\) we have $\pi(G)\leq C\cdot \Delta(G)^2$, for \(C = 2 e^{16}+1\),
where \(\Delta(G)\) denotes the maximum degree of \(G\).
Some results on the path-nonrepetitive chromatic number 
of specific classes of graphs include Thue's Theorem mentioned above~\cite{thue1906unendliche}
that says, in other words, that $\pi(P_n)=3$, for every $n$;
and Currie's result~\cite{currie2002there} that says that $\pi(C_n)=4$, if $n\in \{5,7,9,10,14,17\}$ and $\pi(C_n)=3$ otherwise.
For more results on path-nonrepetitive colorings, we refer the reader to~\cite{Wood2020}.

We consider two variations of the concept above that were introduced in~\cite{BaratWood08},
and answer some questions posed in~\cite{Wood2020}.
More specifically, we study walk-nonrepetitive colorings of cycles
and stroll-nonrepetitive colorings of paths and cycles,
which we define below.

\bigskip\noindent\textbf{Walk-nonrepetitive colorings.}
We say that a walk \(w = v_1 \cdots v_{2t}\) is \emph{boring} if $v_i = v_{i+t}$ for each $i \in \{1,\ldots,t\}$;
otherwise we say that \(w\) is a \emph{nonboring} walk.
Note that every boring walk is repetitive in every coloring \(c\) of \(V(G)\).
Thus, we say that a coloring \(c\) of a graph \(G\) is \emph{walk-nonrepetitive}
if every repetitive walk is boring;
and define the \emph{walk-nonrepetitive chromatic number},
which is denoted by \(\sigma(G)\), to be the minimum \(k\) for which there is a walk-nonrepetitive \(k\)-coloring of \(G\).
Since every path is a nonboring walk, 
every walk-nonrepetitive coloring is a path-nonrepetitive coloring,
and hence \(\pi(G) \leq \sigma(G)\).

In this paper we study walk-nonrepetitive colorings of cycles.
Naturally, \(\sigma(C_3) = 3\).
In 2008, Barát and Wood~\cite{BaratWood08} proved that \(4\leq \sigma(C_n) \leq 5\) for every \(n\geq 4\), and in 2021, Wood~\cite[Open Problem 3.9]{Wood2020} asked whether \(\sigma(C_n) = 4\) for infinitely many \(n\).
In this paper we answer this question showing that for \(n\geq 4\),
we have \(\sigma(C_n) = 4\) unless \(n=5\) or \(n=7\).

\begin{restatable}{theorem}{wnrccycles}
\label{thm:wnrc-cycles}
Let \(n\geq 4\) be a positive integer.
If \(n\notin\{5,7\}\), then $\sigma(C_n) = 4$.
\end{restatable}

For that, we extend to cycles a characterization of walk-nonrepetitive colorings of trees given by Wood~\cite{BaratWood08}:
A coloring \(c\) of a cycle \(C_n\)
is walk-nonrepetitive if and only if \(c\) is a path-nonrepetitive distance-\(2\) coloring (see Lemma~\ref{theorem:characterizing-wnrc-of-cycles}).
A distance-2 coloring is a coloring in which vertices with distance at most 2 are colored with distinct colors (see Section~\ref{sec:wnrc-cycles}).

\bigskip\noindent\textbf{Stroll-nonrepetitive colorings.}
We say that a walk $w = v_1 \cdots v_{2t}$ is a \emph{stroll} if $v_i \neq v_{i+t}$ for each $i\in \{1,\ldots, t\}$,
i.e.,
if no vertex of the first half of \(w\) is repeated
in precisely the same position in the second half of \(w\).
In other words, a walk \(v_1\cdots v_{2t}\) 
fails to be a stroll if \(v_i = v_{i+t}\) for some \(i\in\{1,\ldots,t\}\).
Naturally, any path in \(G\) is a stroll,
and no stroll is a boring walk.
We then say that a coloring $c$ of \(G\) is \emph{stroll-nonrepetitive} if \(c\) has no repetitive stroll, and define the \emph{stroll-nonrepetitive chromatic number},
which is denoted by $\rho (G)$, as the minimum $k$ for which there is a stroll-nonrepetitive $k$-coloring of~\(G\). 
Note that 
every walk-nonrepetitive coloring is a stroll-nonrepetitive coloring, and
every stroll-nonrepetitive coloring is a path-nonrepetitive coloring.
Therefore, we have 

\begin{equation}\label{eq:relation-btw-chromatic_numbers}
    \pi(G)\leq \rho(G)\leq \sigma(G)
\end{equation}

In 2008, K{\"u}ndgen and Pelsmajer~\cite{kundgen2008nonrepetitive} proved that \(\sigma(P_n) \leq 4\) for every \(n\),
from which we conclude that \(\rho(P_n)\leq 4\) for every \(n\);
and in~2021 Ochem proved that \(\rho(P_n) =  4\) for \(n\geq 24\) (see~\cite[Proposition~3.3]{Wood2020}).
A version of this result for large \(n\) was given by Tao~\cite{tao2023non}.
Observe that any proper \(2\)-coloring of \(P_3\)
is stroll-nonrepetitive because
every repetitive walk must repeat the middle vertex in the same position 
in its two halves, and hence is not a stroll.
In this paper, we  modify Ochem's approach in order to settle \(\rho(P_n)\) as follows.

\begin{restatable}{theorem}{snrcpaths}
\label{thm:snrc-paths}
Let \(n\geq 4\) be a positive integer.
Then
\[
\rho(P_n)=\begin{cases}
        3, & \text{if }  n\leq 21\\
        4, & \text{if } n\geq 22
\end{cases}
\]
\end{restatable}

In 2021, Wood~\cite[Open Problem 3.9]{Wood2020} asked whether \(\rho(C_n)\leq 3\) for infinitely many \(n\) or \(\rho(C_n)\leq 4\) for infinitely many \(n\).
Now, suppose \(n \geq 22\).
Since \(C_n\) contains \(P_n\), we have \(\rho(C_n) \geq \rho(P_n)\),
and hence, by Theorem~\ref{thm:snrc-paths}, we have \(\rho(C_n) \geq 4\).
Then, by Theorem~\ref{thm:wnrc-cycles} and Equation~\ref{eq:relation-btw-chromatic_numbers}, 
we have \(\rho(C_n) = 4\),
which answers Wood's question.
In fact, we can obtain the following slightly stronger result.

\begin{restatable}{theorem}{snrcycles}
\label{thm:snrc-cycles}
Let \(n\geq 3\) be a positive integer. Then
\[
\rho(C_n)=\begin{cases}
        3, & \text{if }  n \in\{3,4,6,8\}\\
        4, & \text{if } n \notin \{3,4,6,8\}
\end{cases}
\]
\end{restatable}

\section{The walk-nonrepetitive chromatic number of cycles}\label{sec:wnrc-cycles}

Given a graph \(G\), we say that a coloring \(c\colon V(G) \to \{1,\ldots,k\}\)
is a \emph{distance-\(2\) coloring} if \(c(u) \neq c(v)\) for every \(u,v\in V(G)\)
for which \(dist(u,v) \leq 2\).
In other words, \(c\) is a distance-\(2\) coloring if
\(c(u) \neq c(v)\) whenever \(uv\in E(G)\) or \(u\) and \(v\) have a common neighbor.
In this section we prove Theorem~\ref{thm:wnrc-cycles}, which settles 
the walk-nonrepetitive chromatic number of cycles.
In~2008, Barát and Wood proved that if \(G\) is a tree, 
then a coloring \(c\) of \(G\) is walk-nonrepetitive if and only if \(c\) is a distance-\(2\) path-nonrepetitive   coloring~\cite[Theorem~3.1]{BaratWood08}.
The next result shows that such a characterization holds also when \(G\) is a cycle.

\begin{lemma}\label{theorem:characterizing-wnrc-of-cycles}
Let \(c\) be a coloring of a cycle \(C\).
Then \(c\) is walk-nonrepetitive if and only if \(c\) is a distance-\(2\) path-nonrepetitive coloring.
\end{lemma}

\begin{proof}
First, suppose that \(c\) is walk-nonrepetitive.
Since every path is a nonboring walk, \(c\) is path-nonrepetitive. 
Now, if a vertex $v$ has two neighbors $u$ and $w$ colored with the same color, 
then $uvwv$ is a repetitive nonboring walk, a contradiction. 
This implies that $c$ is a distance-2 coloring.
In what follows, we prove the converse, i.e., 
that if \(c\) is a path-nonrepetitive distance-\(2\) coloring, 
then \(c\) is a walk-nonrepetitive coloring.

Suppose, for a contradiction, that \(c\) is not a walk-nonrepetitive coloring.
We claim that if \(w = v_1 \cdots v_{2t}\) is a nonboring walk with \(t\in\{1,2\}\),
then \(c\) does not induce a repetitive coloring of \(w\).
Indeed, if \(t = 1\), then \(v_1 \neq v_2\), and since \(c\) is a proper coloring, 
we have \(c(v_1) \neq c(v_2)\), as desired.
Thus, assume \(t = 2\).
Since \(w\) is a nonboring walk, we have either \(v_1 \neq v_3\) or \(v_2\neq v_4\).
We may assume, by symmetry, that \(v_1\neq v_3\).
Since \(c\) is a distance-\(2\) coloring, we have \(c(v_1) \neq c(v_3)\),
and hence \(c\) does not induce a repetitive coloring of \(w\), as desired.
Now, let \(w = v_1 \cdots v_{2t}\) be a nonboring walk with minimum order
in which~\(c\) induces a repetitive coloring.
As shown above, we must have \(t > 2\).
Also, by the minimality of \(t\), \(c\) does not induce a repetitive coloring of \(w'\)
for every nonboring walk \(w' = v'_1 \cdots v'_{2t'}\) with \(t' < t\).

Since \(c\) is a path-nonrepetitive coloring, we may assume that \(w\) is not a path.
Suppose now that for some \(i\) we have \(c(v_{i-1}) = c(v_{i+1})\).
Since \(c\) induces a repetitive coloring of \(w\), we may assume that \(i \leq t+1\).
If  \(i < t\), then \(c(v_{j-1}) = c(v_{j+1})\) for \(j = i +t\),
and since \(c\) is a distance-\(2\) coloring, we have \(v_{i-1} = v_{i+1}\),
and \(v_{j-1} = v_{j+1}\)
but then \(c\) induces a repetitive coloring of 
\(w' = v_1\cdots v_{i-1}v_{i+2} \cdots v_{j-1}v_{j+2} \cdots v_{2t}\).
Clearly \(w'\) is a walk with \(2t - 4\) vertices.
By the minimality of \(w\), \(w'\) is a boring walk.
Then \(v_{i+1} = v_{i-1} = v_{j-1} = v_{j+1}\),
and since \(c(v_i) = c(v_j)\), and \(v_{i-1}\) has only one neighbor colored with \(c(v_i)\),
we have \(v_i = v_j\),
and hence \(w\) is a boring walk,
a contradiction.
Thus we have either \(i = t\) or \(i = t+1\).
In this case \(c\) induces a repetitive coloring of 
\(w' = v_2\cdots v_{t-1}v_{t+2} \cdots v_{2t-1}\).
Again, \(w'\) is a boring walk, and, since \(v_2 = v_{t+2}\) has only one neighbor colored with color \(c(v_1)\), and \(v_{t-1} = v_{2t-1}\) has only one neighbor colored
with color \(c(v_t)\), then \(w\) is a boring walk, a contradiction.

Therefore, we have \(c(v_{i-1})\neq c(v_{i+1})\) for every \(i \in \{2,\ldots, 2t-1\}\),
and hence \(v_{i-1}\neq v_{i+1}\) for every \(i \in \{2,\ldots, 2t-1\}\). 
In other words, the walk never ``changes direction''.
Therefore, since \(w\) is not a path, \(v_1\) appears twice in \(w\).
Let \(\ell = |C|\), and let \(k\geq 1\) be the maximum integer for which \(s = 2t - k\cdot\ell \geq 0\), i.e., \(k\) is the number of complete consecutive copies of \(C\) in \(w\),
and \(s < \ell\) is the number of remaining vertices that do not form a complete copy of \(C\). 
Therefore, for $i\in\{1,\ldots,s\}$, 
we have \(v_i = v_{i+\ell} = \cdots = v_{i+k\cdot \ell} = v_{i + 2t-s}\),
where in the last equality we used that \(k\cdot \ell= 2t-s\).
Thus $c(v_{i})=c(v_{t+i})$ and $c(v_{2t-s+i})=c(v_{t-s+i})$.
Now, consider the walks

\vspace{-20pt}
\begin{align*}
w_1 &= v_1 \cdots v_s &
w_2 &= v_{s+1} \cdots v_{t-s} &
w_3 &= v_{t-s+1} \cdots v_t \\
w_4 &= v_{t+1} \cdots v_{t+s} &
w_5 &= v_{t+s+1} \cdots v_{2t-s} &
w_6 &= v_{2t-s+1} \cdots v_{2t}
\end{align*}
which may have common vertices.
Since \(c\) induces a repetitive coloring of \(w\),
\(c\) induces the same coloring in \(w_{r}\) and \(w_{r+3}\) for \(r\in\{1,2,3\}\).
Now, since \(v_i  = v_{i + 2t-s}\) for every \(i\in\{1,\ldots,s\}\),
\(c\) induces the same coloring in \(w_1\) and \(w_6\),
and, consequently, \(c\) induces the same coloring in \(w_3\) and \(w_4\).
Finally, let \(w' = w_3w_4\) be the walk consisting of the \(2s\) central vertices of \(w\).
Observe that since $k\geq 1$, we have $2t=k\cdot\ell +s \geq \ell+s > 2s$.
But then \(c\) is a repetitive coloring of \(w'\),
which is a walk with \(2s < 2t\) vertices. 
Moreover, \(v_{t-s+1} \neq v_{t+1}\) because \(s < \ell\),
and hence \(w'\) is nonboring, a contradiction.
\end{proof}

We observe that the characterization above cannot be extended to more general unicyclic graphs (see Figure~\ref{fig:unicyclic}).

\begin{figure}[H]
	\begin{center}
		\begin{tikzpicture}
		[scale=1,auto=left]   
			\tikzstyle{every node}=[circle, draw, fill=white,
			inner sep=2.3pt, minimum width=2.3pt]

        \node[fill = green, label=below:\small $2$]  (b) at (1,0) {};
        \node[fill = red, label=left:\small $1$]  (a) at ($(b) + (150:1)$) {};
        \node[fill = blue, label=left:\small $3$]  (c) at ($(b) + (-150:1)$) {};
        \node[fill = yellow, label=below:\small $4$]  (d) at (2,0) {};
        \node[fill = red,label=below:\small $1$]  (e) at (3,0) {};
        \node[fill = green,label=below:\small $2$]  (f) at (4,0) {};
        \node[fill = blue,label=below:\small $3$]  (g) at (5,0) {};
        \node[fill = red,label=below:\small $1$]  (h) at (6,0) {};
        \node[fill = green,label=below:\small $2$]  (i) at (7,0) {};
        \node[fill = yellow,label=below:\small $4$]  (j) at (8,0) {};
        
		\foreach \from/\to in {a/b,a/c,b/c,b/d,d/e,e/f,f/g,g/h,h/i,i/j},  
		\draw (\from) -- (\to); 
		 
		\end{tikzpicture}   
	\end{center}
	\caption{A unicyclic graph with a distance-$2$ path-nonrepetitive coloring that is not a walk-nonrepetitive coloring.}
	\label{fig:unicyclic}
	\end{figure}
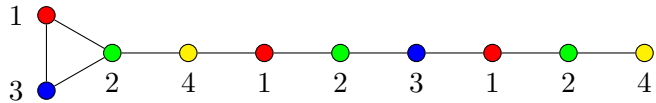	

In what follows, we prove the main result of this section.
For that, we need the following result proved by Currie~\cite{currie2002there} 
on the path-nonrepetitive chromatic number of cycles.

\begin{theorem}[Currie, 2002]\label{thm:currie-path-nonrep-cycle}

Let \(n\geq 3\) be a positive integer.
Then
\[
\pi(C_n)=\begin{cases}
        4, & \text{if }  n \in \{5,7,9,10,14,17\}\\
        3, & \text{otherwise}
\end{cases}
\]
\end{theorem}

As the base cases of our result, we show in Table~\ref{TableWalkNonRepetitiveCycleGraphs}
minimum walk-nonrepetitive colorings of cycles of order at most \(21\)
found with the aid of a computer.
For that, we used a naive Integer Linear Programming formulation for distance-\(2\) colorings
with additional restrictions induced by the paths with even order in \(C_n\)\footnote{A linear program for finding walk-nonrepetitive colorings of cycles with a given number of colors in Sage~\cite{sagemath} is available at \href{https://www.cos.ufrj.br/~fbotler/codes/wnrc-cycles.ipynb}{https://www.cos.ufrj.br/~fbotler/codes/wnrc-cycles.ipynb}}.
Lemma~\ref{theorem:characterizing-wnrc-of-cycles} guarantees that such restrictions
are enough to find walk-nonrepetitive colorings of such graphs.

\begin{proposition}\label{prop:base-case-walk-non-repetitive-cycles}
If \(4\leq n \leq 21\) and \(n \notin \{5,7\}\), then \(\sigma(C_n) = 4\). If \(n \in\{5,7\}\), then \(\sigma(C_n) = 5\).
\end{proposition}

\begin{table}
     \centering
     \begin{tabular}{|c|c|l|}
     \hline
         $n$ & $\sigma(C_n)$ & Sequence of colors of the vertices \\
         \hline
          4& 4 &  $[1, 2, 3, 4]$ \\
          5& 5&   $[1, 2, 3, 4, 5]$\\
          6& 4&  $[1, 2, 3, 4, 2, 3]$\\
          7& 5&  $[1, 2, 3, 4, 2, 5, 3]$ \\
          8& 4&  $[1, 2, 3, 4, 1, 2, 4, 3]$ \\
          9& 4&  $[1, 2, 3, 4, 1, 3, 2, 4, 3]$ \\
          10& 4&  $[1, 2, 3, 4, 1, 2, 4, 3, 2, 4]$ \\
          11& 4& $[1, 2, 3, 1, 2, 4, 3, 2, 1, 3, 4]$  \\
          12& 4&  $[1, 2, 3, 4, 1, 2, 4, 3, 1, 4, 2, 3]$ \\
          13& 4&  $[1, 2, 3, 4, 1, 2, 4, 3, 1, 4, 3, 2, 4]$ \\
          14& 4&  $[1, 2, 3, 1, 2, 4, 1, 2, 3, 1, 4, 2, 3, 4]$ \\
          15& 4& $[1, 2, 3, 1, 4, 2, 1, 4, 3, 1, 4, 2, 1, 3, 4]$ \\
          16& 4& $[1, 2, 3, 1, 4, 2, 3, 4, 2, 1, 4, 2, 3, 1, 4, 3]$  \\
          17& 4& $[1, 2, 3, 1, 4, 3, 2, 4, 1, 3, 4, 2, 3, 1, 2, 3, 4]$  \\
          18& 4&  $[1, 2, 3, 1, 4, 2, 3, 4, 1, 3, 2, 4, 1, 2, 3, 1, 4, 3]$ \\
          19& 4&  $[1, 2, 3, 4, 2, 3, 1, 2, 4, 1, 2, 3, 4, 1, 3, 2, 1, 3, 4]$ \\
          20& 4& $[1, 2, 3, 1, 4, 2, 1, 4, 3, 2, 4, 1, 2, 3, 1, 2, 4, 1, 3, 4]$  \\
          21& 4&  $[1, 2, 3, 4, 1, 2, 4, 1, 3, 4, 2, 1, 3, 2, 4, 3, 2, 1, 4, 2, 3]$ \\
          \hline
     \end{tabular}
     \caption{Walk-nonrepetitive coloring of cycles of order at most 21.}
     \label{TableWalkNonRepetitiveCycleGraphs}
 \end{table}

Let \(G\) be a path or cycle and let \(c\colon V(G)\to\{1,2,3\}\)
be a path-nonrepetitive coloring of \(G\).
We say that a vertex \(v\in V(G)\) with \(d(v) = 2\) is a \emph{symmetrical} (with respect to \(c\))
if their neighbors are colored with the same color,
otherwise, we say that \(v\) is a \emph{asymmetrical}.
We also say that an edge \(e\in E(G)\) is a \emph{bad edge} (with respect to \(c\))
if \(e\) is incident to a symmetrical vertex,
and that \(e\) is a \emph{good edge} otherwise (see Figure~\ref{fig:goodandbadedge}).
The following lemma on path-nonrepetitive colorings of even cycles is used in the proof of Theorem~\ref{thm:wnrc-cycles}.

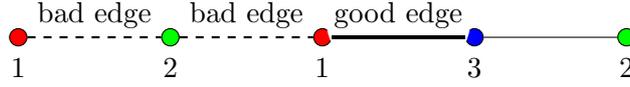
\begin{figure}[H]
	\begin{center}
		\begin{tikzpicture}
		[scale=1,auto=left]   
			\tikzstyle{every node}=[circle, draw, fill=white,
			inner sep=2.3pt, minimum width=2.3pt]
		\node[fill = red, label=below:\small $1$]  (a) at (0,0) {};
        \node[fill = green, label=below:\small $2$]  (b) at (2,0) {};
        \node[fill = red, label=below:\small $1$]  (c) at (4,0) {};
        \node[fill = blue, label=below:\small $3$]  (d) at (6,0) {};
        \node[fill = green, label=below:\small $2$]  (e) at (8,0) {};
        \node[draw = none] (f) at (1,0.3) {\small{bad edge}};
        \node[draw = none] (g) at (3,0.3) {\small{bad edge}};
        \node[draw = none] (h) at (5,0.3) {\small{good edge}};

		\foreach \from/\to in {c/d,d/e}  
		\draw (\from) -- (\to); 
        \draw[thick,dashed] (a) -- (b);
        \draw[thick, dashed] (b) -- (c);
        \draw[line width = 1.5] (c) -- (d);		 
		 
		\end{tikzpicture}   
	\end{center}
	\caption{A bad edge and a good edge depicted as, respectively, a dashed line and thick straight line: The leftmost green (color 2) vertex is a symmetrical vertex, while the middle red (color 1) vertex and the blue (color 3) vertex are asymmetrical vertices.} 
	\label{fig:goodandbadedge}
	\end{figure}

\begin{lemma}\label{lemma:good-edges}
    Let \(k\) be a positive integer.
    Let \(C = v_1v_2 \cdots v_{2k} v_1\) be a cycle of length \(2k\),
    and \(c\colon V(C)\to\{1,2,3\}\) be a path-nonrepetitive coloring of \(C\).
    If \(k\geq 8\),
    then there are at least two good edges in \(\{v_1v_2,v_3v_4,\ldots, v_{2k-1}v_{2k}\}\).
\end{lemma}

\begin{proof}
We shall use the following claims.

\begin{claim}\label{claim:around-bad-vertex}
Let \(w_1,w_2,w_3,w_4,w_5\) be a sequence of consecutive vertices in \(C\).
If \(w_3\) is a symmetrical vertex, then \(c(w_1) = c(w_5)\).
Furthermore, \(w_2\) and \(w_4\) are asymmetrical vertices.
\end{claim}

\begin{proof}
Suppose, without loss of generality that \(c(w_3) = 3\),
and \(c(w_2) = c(w_4) = 2\).
If \(c(w_1) = 2\), then \(w_1w_2\) is a repetitive path;
and if \(c(w_1) = 3\), then \(w_1w_2w_3w_4\) is a repetitive path.
Therefore, \(c(w_1) = 1\).
Analogously, \(c(w_5) = 1\), as desired.
\end{proof}

\begin{claim}\label{claim:no-three-consecutive-bad-edges}
Let \(w_1,w_2,w_3,w_4,w_5,w_6,w_7,w_8\) a sequence of consecutive vertices in \(C\)
in which \(w_2\) is a symmetrical vertex.
Then at least one between \(w_3w_4,w_5w_6\) is a good edge.
\end{claim}

\begin{proof}
Suppose, without loss of generality, \(c(w_1) = c(w_3) = 1\) and \(c(w_2) = 2\),
and suppose that \(w_3w_4\) and \(w_5w_6\) are bad edges.
By Claim~\ref{claim:around-bad-vertex},  \(w_3\) is an asymmetrical vertex,
and hence \(c(w_4) = 3\).
Since \(w_3w_4\) is a bad edge, \(w_4\) must be a symmetrical vertex,
which implies that \(c(w_5) = 1\).
Again, by Claim~\ref{claim:around-bad-vertex}, \(w_5\) is an asymmetrical vertex,
and hence \(c(w_6) = 2\);
and since \(w_5w_6\) is a bad edge, \(w_6\) must be a symmetrical vertex,
which implies \(c(w_7) = 1\).
Finally, by Claim~\ref{claim:around-bad-vertex}, \(c(w_8) = 3\),
but \(w_1 \cdots w_8\) is a repetitive path,
a contradiction.
\end{proof}

Let \(e_i = v_{2i-1}v_{2i}\),
and put \(S=\{e_i : i = 1,\ldots,k\}\).
By relabeling \(e_i\) either by shifting or reversing direction, we may assume, without loss of generality, that \(e_5\) is a bad edge
and \(v_{10}\) is a symmetrical vertex.
By Claim~\ref{claim:no-three-consecutive-bad-edges}, \(e_6\) or \(e_7\) is a good edge.
If \(e_3\) is a good edge, then we are done.
Thus, we may assume that \(e_3\) is a bad edge.
If \(v_5\) is a symmetrical vertex, then Claim~\ref{claim:no-three-consecutive-bad-edges}
tell us that \(e_1\) or \(e_2\) is a good edge.
Since \(k \geq 8\), we have \(\{e_1,e_2\}\cap\{e_6,e_7\} = \emptyset\),
and we are done.
Thus, we may assume \(v_6\) is a symmetrical vertex.
Now, by Claim~\ref{claim:no-three-consecutive-bad-edges},
since \(e_5\) is a bad edge, \(e_4\) must be a good edge.
This concludes the proof.
\end{proof}

Now, we may prove the main result of this section.

\wnrccycles*

\begin{proof}
Recall that, by Barát and Wood~\cite{BaratWood08}, \(4\leq \sigma(C_n) \leq 5\) for every \(n\geq 4\). 
Moreover, by Proposition~\ref{prop:base-case-walk-non-repetitive-cycles}, we may assume \(n>21\).
Let \(k = \lceil n/3\rceil\) and note that \(k\geq 8\).
Let \(C' = v_1 \cdots v_{2k}v_1\) be a cycle of length \(2k\).
By Theorem~\ref{thm:currie-path-nonrep-cycle},
\(C'\) admits a path-nonrepetitive \(3\)-coloring, say \(c'\colon V(C') \to \{1,2,3\}\).
Let \(e_i = v_{2i-1}v_{2i}\),
and put \(S=\{e_i : i = 1,\ldots,k\}\).
By Lemma~\ref{lemma:good-edges}, there are at least two good edges in \(S\).

Let \(m\in\{0,1,2\}\) be such that \(m=3k-n\) and put \(k' = k-m\).
Now, let \(S'\subseteq S\) be a set obtained from \(S\) 
by removing \(m\) of its good edges.
Let \(C\) be the cycle obtained from \(C'\) 
by subdividing each edge in \(S'\) once, i.e., by replacing each such edge by a path with two edges.
Let \(u_1,\ldots,u_{k'}\) be the vertices added due to the subdivisions,
and note that \(C\) has length \(2k + k' = 3k -m = n\).
Let \(c\colon V(C)\to\{1,2,3,4\}\) be a coloring of \(C\)
such that \(c(v) = c'(v)\) if \(v\in V(C')\),
and \(c(v) = 4\) otherwise,
i.e., if \(v = u_j\) for some \(j\) (see Figure~\ref{fig:techniqueoftheorem1}).

We claim that \(c\) is a walk-nonrepetitive coloring of \(C\).
Clearly, any two adjacent vertices are colored with different colors.
Since $c'$ is a proper coloring, the two neighbors (in \(C\)) 
of $u_j$ are colored with different colors,
for every \(j\in \{1,\ldots,k'\}\).
If \(v\in V(C')\) is an asymmetrical vertex, then either \(v\) has a neighbor (in \(C\)) with color \(4\)
and another colored with color in \(\{1,2,3\}\),
or \(v\) has two distinct neighbors with distinct colors in \(\{1,2,3\}\).
Moreover, since every symmetrical vertex of \(C'\) is covered by \(S'\),
if $v\in V(C')$ is a symmetrical vertex,
then \(v\) has exactly one neighbor in $V(C)\setminus V(C')$,
and hence has precisely one neighbor with color in \(\{1,2,3\}\),
and one neighbor with color \(4\).
Therefore $c$ is a distance-$2$ coloring.

Now, suppose that there is a path $P=w_1\cdots w_{2t}$
for which \(c\) is a repetitive coloring.
By supressing in \(P\) the vertices in \(V(C)\setminus V(C')\),
i.e., replacing the path \(v_{2i-1}u_jv_{2i}\) by the edge \(v_{2i-1}v_{2i}\),
we obtain a path \(P'\) in \(C'\)
for which \(c'\) is a repetitive coloring,
a contradiction.
Therefore, \(c\) is a path-nonrepetitive coloring,
and hence, by Theorem~\ref{theorem:characterizing-wnrc-of-cycles},
\(c\) is a walk nonrepetitive \(4\)-coloring of \(C\)
as desired.
\end{proof}

\begin{figure}[H]
	\begin{center}
		\begin{tikzpicture}
		[scale=1,auto=left]   
			\tikzstyle{every node}=[circle, draw, fill=white,
			inner sep=2.3pt, minimum width=2.3pt]
		\node[fill = red, label=90:\small $1$]  (a) at (2.25246,3.17043) {};
  	\node[fill = green, label=135:\small $2$]  (b) at (0.96027,2.41317) {};
    	\node[fill = blue, label=195:\small $3$]  (c) at (0.96224,0.90956) {};
        \node[fill = red, label=-90:\small $1$]  (d) at (2.26959,0.16616) {};
        \node[fill = blue, label=-45:\small $3$]  (e) at (3.56951,0.91938) {};
        \node[fill = green, label=45:\small $2$]  (f) at (3.55989,2.42054) {};

        \node[fill = red, label=90:\small $1$]  (a1) at (6.45246,3.17043) {};
  	\node[fill = green, label=135:\small $2$]  (a2) at (5.16027,2.41317) {};
    	\node[fill = blue, label=195:\small $3$]  (a3) at (5.16024,0.90956) {};
        \node[fill = red, label=-90:\small $1$]  (a4) at (6.46959,0.16616) {};
        \node[fill = blue, label=-45:\small $3$]  (a5) at (7.76951,0.91938) {};
        \node[fill = green, label=45:\small $2$]  (a6) at (7.75989,2.42054) {};

        \node[fill = yellow, label=180:\small $4$]  (g) at (5.161255,1.6613385) {};
        \node[fill = yellow, label=-45:\small $4$]  (h) at (7.11955,0.54277) {};
        \node[fill = yellow, label=45:\small $4$]  (i) at (7.106175,2.795485) {};

		\foreach \from/\to in {a/b,b/c,c/d,d/e,e/f,a/f,a1/a2,a2/g,g/a3,a3/a4,a4/h,h/a5,a5/a6,a6/i,i/a1},  
		\draw (\from) -- (\to); 
		\end{tikzpicture}   
	\end{center}
	\caption{A walk-nonrepetitive coloring of $C_9$ obtained from a path-nonrepetitive $3$-coloring of $C_6$ by subdividing three edges and coloring the new vertices with yellow.}
	\label{fig:techniqueoftheorem1}
	\end{figure}
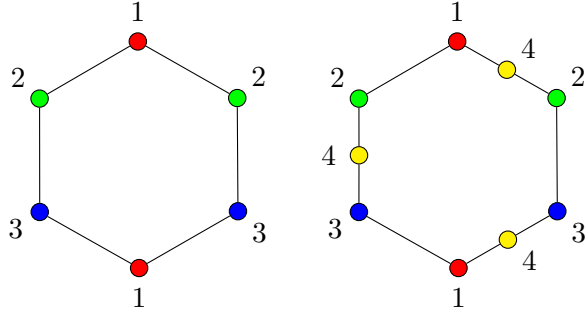

\section{Stroll-nonrepetitive colorings of paths and cycles} 

Recall that a walk \(w = v_1\cdots v_{2t}\) is a \emph{stroll} 
if \(v_i \neq v_{i+t}\) for every \(i\in\{1,\ldots,t\}\),
and that a coloring \(c\colon V(G)\to \{1,\ldots,k\}\) of the vertices of a graph \(G\)
is \emph{stroll-nonrepetitive} if 
for every stroll \(v_1 \cdots v_{2t}\)
the sequence \(c(v_1) \cdots c(v_{2t})\)
is nonrepetitive.
In other words, 
\(c\) is a stroll-nonrepetitive coloring if the sequence of colors induced by a walk \(v_1\cdots v_{2t}\) is repetitive only when \(v_i = v_{i+t}\) for some \(i\in\{1,\ldots,t\}\).
Finally, recall that the stroll chromatic number of \(G\) is denoted by \(\rho(G)\).

In this section, we prove Theorem~\ref{thm:snrc-paths}, 
which says that when \(n\geq 4\), we have \(\rho(P_n) = 3\) if \(n\leq 21\),
and \(\rho(P_n) = 4\) otherwise.
We use the following observation. 
Suppose we color \(P_n\) so that color \(1\) appears only once, say on a vertex \(u\).
If \(w = v_1 \cdots v_{2t}\) is a repetitive walk that contains \(u\), say \(v_i = u\), where \(1\leq i\leq t\), then we have \(v_{i+t} = u\),
and hence \(w\) is not a stroll.
A consequence of this observation is that every proper coloring of \(P_3\) is a stroll-nonrepetitive coloring.
The following result presents the stroll chromatic number of paths of order at most \(21\).

\begin{proposition}\label{prop:small-paths-snc}
If \(4\leq n\leq 21\), then \(\rho(P_n) = 3\)
\end{proposition}

\begin{proof}
Since every path of order at most \(20\) is a subpath of \(P_{21}\),
it suffices to prove that \(\rho(P_{21}) = 3\).
Let \(P_{21} = v_1\cdots v_{21}\), and consider the coloring \(c\colon V(P_{21}) \to \{1,2,3\}\)
such that 
\(
    c(v_1) c(v_2) \cdots c(v_{21}) = 121312321323123213121
\)
(see Figure~\ref{fig:P21-snc}).
We prove that \(c\) is stroll-nonrepetitive.
Suppose, for a contradiction, that there is a stroll $W=w_1 \cdots w_{2t}$ in which \(c\) induces a repetitive coloring.
We prove a few claims.

 \begin{figure}[h]
 	\begin{center}
 		\begin{tikzpicture}
 		[scale=.75,auto=left]   
 			\tikzstyle{every node}=[circle, draw, fill=white,
 			inner sep=2.3pt, minimum width=2.3pt]
 		\node[fill = red,label=above: $1$]  (a) at (0,0) {};   
 		\node[fill = green,label=above: $2$]  (b)  at (1,0)  {};
 		\node[fill = red,label=above: $1$]  (c) at (2,0)  {};
 		\node[fill = blue,label=above: $3$]  (d) at (3,0)  {};
 		\node[fill = red,label=above: $1$]  (e) at (4,0)  {};
 		\node[fill = green,label=above: $2$]  (f) at (5,0)  {};
 		\node[fill = blue,label=above: $3$]  (g) at (6,0)  {};
 		\node[fill = green,label=above: $2$]  (h) at (7,0)  {};
 		\node[fill = red,label=above: $1$]  (i) at (8,0)  {};
 		\node[fill = blue,label=above: $3$]  (j) at (9,0)  {};
 		\node[fill = green,label=above: $2$]  (k) at (10,0)  {};
 		\node[fill = blue,label=above: $3$]  (l) at (11,0)  {};
 		\node[fill = red,label=above: $1$]  (m) at (12,0)  {};
 		\node[fill = green,label=above: $2$]  (n) at (13,0)  {};
 		\node[fill = blue,label=above: $3$]  (o) at (14,0)  {};
 		\node[fill = green,label=above: $2$]  (p) at (15,0)  {};
 		\node[fill = red,label=above: $1$]  (q) at (16,0)  {};
 		\node[fill = blue,label=above: $3$]  (r) at (17,0)  {};
 		\node[fill = red,label=above: $1$]  (s) at (18,0)  {};
 		\node[fill = green,label=above: $2$]  (t) at (19,0)  {};
 		\node[fill = red,label=above: $1$]  (u) at (20,0)  {};

 		\foreach \from/\to in {a/b,b/c,c/d,d/e,e/f,f/g,g/h,h/i,i/j,j/k,k/l,l/m,m/n,n/o,o/p,p/q,q/r,r/s,s/t,t/u} 
 		\draw (\from) -- (\to);

 		\end{tikzpicture}   
 	\end{center}
 	\caption{A stroll-nonrepetitive coloring of $P_{21}$.} 
 	\label{fig:P21-snc}
 	\end{figure}
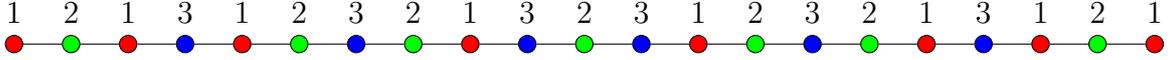

\begin{claim}\label{claim:v11-notin-V(W)}
\(v_{11} \notin V(W)\).
\end{claim}

\begin{proof}
Suppose that \(w_i = v_{11}\) for some \(i\in\{1,\ldots, t\}\).
We claim that, for \(j\in\{1,\ldots, t\}\), either \(w_j = v_{11}\) whenever \(j\) is odd,
or \(w_j = v_{11}\) whenever \(j\) is even.
For that, we prove that if  \(i+2\leq t\), then \(w_{i+2} = v_{11}\),
and if \(i-2\geq 1\), then \(w_{i-2} = v_{11}\).
Suppose that \(i+2\leq t\),
and note that \(c(v_{11}) = 2\) and \(c(v_{10}) = c(v_{12}) = 3\).
Since \(w_{i+t} \neq w_i\), we must have $w_{i+t}\in\{v_6,v_8,v_{14},v_{16}\}$.
Since \(w_{i+1}\in\{v_{10},v_{12}\}\), we have \(c(w_{i+t+1}) = c(w_{i+1}) = 3\),
and hence \(w_{i+t+1}\in \{v_7,v_{15}\}\) because \(v_7\) (resp. \(v_{15}\)) is the only neighbor of \(v_6\) and \(v_8\) (resp.~\(v_{14}\) and~\(v_{16}\)) that is colored with~\(3\).
Now, since \(w_{i+t+1}\in \{v_7,v_{15}\}\) and 
\(c(v_6) = c(v_8) = c(v_{14}) = c(v_{16}) = 2\),
we must have \(c(w_{i+t+2}) = 2\).
Since \(c\) induces a repetitive coloring in \(W\),
we must have \(c(w_{i+2}) = 2\),
and since \(w_{i+1}\in\{v_{10},v_{12}\}\) and  the only neighbor of 
\(v_{10}\) and \(v_{12}\) colored with \(2\) is  \(v_{11}\),
we have \(w_{i+2} = v_{11}\) as desired.
Analogously, if \(i-2\geq 1\), then \(c(w_{i-1})= 3\) and \(w_{i-2} = v_{11}\).

We conclude that \(w_1,\ldots,w_t \in \{v_{10},v_{11},v_{12}\}\).
Then \(w_{t-1} = v_{11}\) or \(w_{t} = v_{11}\)
and \(w_{t+1}\) or \(w_{t+2}\) is in \(\{v_6,v_8,v_{14},v_{16}\}\).
Then a subsequence \(w'_1w'_2\cdots w'_\ell\) of \(w^* = w_{t-1}w_tw_{t+1}w_{t+2}\) starts in \(v_{11}\) and ends in a 
vertex in \(\{v_6,v_8,v_{14},v_{16}\}\).
This can only happen if \(w'_1w'_2\cdots w'_\ell = w^*\) and
either \(w^* = v_{11}v_{10}v_9v_8\) 
or \(w^* = v_{11}v_{12}v_{13}v_{14}\).
By symmetry, assume that \(w^* = v_{11}v_{12}v_{13}v_{14}\).
Since \(c(w_{t+1}) = c(v_{13}) = 1\), we have \(c(w_1) = 1\),
a contradiction.
\end{proof}

By Claim~\ref{claim:v11-notin-V(W)}, either \(V(W)\subseteq \{v_1,\ldots,v_{10}\}\) of \(V(W)\subseteq\{v_{12},\ldots,v_{21}\}\).
By symmetry, assume \(V(W)\subseteq \{v_1,\ldots,v_{10}\}\).

\begin{claim}\label{claim:v7-notin-V(W)}
\(v_{7} \notin V(W)\).
\end{claim}

\begin{proof}
Suppose that \(v_7\in V(W)\),
and let \(w_i = v_7\) for \(i\in \{1,\ldots, t\}\).
If \(i < t\), then \(c(w_{i+1}) = 2\),
and hence \(c(w_{i+t+1}) = 2\).
Now, since \(c(w_{i+t}) = c(w_i) = 3\) and \(c(w_{i+t+1}) = 2\),
we have \(w_{i+t} = v_7\), a contradiction.
Otherwise, if \(i = t\), then \(c(w_{i-1}) = 2\),
and hence \(c(w_{i+t-1}) = 2\).
Now, since \(c(w_{i+t}) = 3\) and \(c(w_{i+t-1}) = 2\),
we have \(w_{i+t} = v_7\), a contradiction.
\end{proof}

By Claim~\ref{claim:v7-notin-V(W)}, we conclude that \(V(W)\subseteq\{v_8,v_9,v_{10}\}\) or \(V(W)\subseteq\{v_1,\ldots,v_6\}\).
Since three vertices cannot contain a repetitive stroll,
we have \(V(W)\subseteq\{v_1,\ldots,v_6\}\);
and since \(v_4\) is the unique vertex in \(\{v_1,\ldots,v_6\}\) colored with \(3\),
we have \(v_4\notin V(W)\), and \(V(W)\subseteq\{v_1,v_2,v_3\}\), a contradiction.
\end{proof}

In what follows, we prove that \(\rho(P_{n}) = 4\) for every \(n\geq 22\).
Recall that, given a \(3\)-coloring \(c\) of a path or a cycle \(G\),
and a vertex \(v\) such that \(d(v) = 2\),
we say that \(v\) is \emph{symmetrical}
if its neighbors have the same color,
and \emph{asymmetrical} otherwise.
If \(G\) is a path (resp. cycle) \(v_1\cdots v_n\),
then we define the \emph{\(SA\)-sequence} of \(c\) 
as the sequence \(a_2 \cdots a_{n-1}\) (resp. \(a_1 \cdots a_n\))
in which \(a_i = S\) if \(v_i\) is symmetrical,
and \(a_i = A\) otherwise.
Note that a proper \(3\)-coloring \(c\) of \(G\) is characterized by the triple consisting of
(i) the color of \(v_1\), (ii) the color of \(v_2\),
and (iii) the \(SA\)-sequence of \(c\).
We assume, without loss of generality, 
that \(c(v_1) = 1\), \(c(v_2) = 2\),
and hence \(c\) is characterized solely by its \(SA\)-sequence.

Given an \(SA\)-sequence \(R\) of a \(3\)-coloring of a path or cycle \(G\), a \emph{subword} of \(R\) is any sequence of consecutive elements of \(R\). 
To prove that \(\rho(P_n) = 4\) when \(n \geq 22\),
we use the following lemma, 
that shows that some \(SA\)-sequences induce repetitive strolls,
and hence are forbidden subwords in any \(SA\)-sequence of a stroll-nonrepetitive coloring of a path.
In what follows, let \(\mathcal{H} = \{SS,AAAA,ASASA,AASAASAA,AAASAAASAAA\}\).
We say that a word \(W\) on the alphabet \(\{S,A\}\) is \emph{\(\mathcal{H}\)-free} if no subword of \(W\) is in \(\mathcal{H}\).

\begin{lemma}\label{3-colStrollNonRepetitiveBlocks}

The \(SA\)-sequence of a stroll-nonrepetitive \(3\)-coloring of any path is \(\mathcal{H}\)-free.
\end{lemma}

\begin{proof}
Let \(c\) be a stroll-nonrepetitive \(3\)-coloring of a path \(P\), 
and let \(R = a_2 \cdots a_{n-1}\) be a subword of the \(SA\)-sequence of \(c\).
Let \(P' = v_1 \cdots v_n\) be the subpath of \(P\) whose \(SA\)-sequence is \(R\).
Assume, without loss of generality, that \(c(v_1) = 1\) and \(c(v_2) = 2\).
We present a repetitive stroll induced by each of these sequences.
If \(R = SS\) (resp. \(R = AAAA\)),
then \(v_1v_2v_3v_4\) is a stroll that induces the repetitive sequence \(1 2 1 2\) (resp. \(v_1 \cdots v_6\) is a stroll that induces \(1 2 3 1 2 3\)),
a contradiction.
Now, if \(R = ASASA\), then \(c(v_1) \cdots c(v_7) = 1232123\) and
$v_1v_2v_3v_4v_5v_6v_7v_6$ is a stroll that induces the repetitive sequence $12321232$;
if \(R = AASAASAA\), then \(c(v_1) \cdots c(v_{10}) = 1231321231\) and
$v_2v_1v_2v_3v_4v_5v_6v_7v_8v_9v_{10}v_9$ is a stroll that induces the repetitive sequence $212313212313$;
and if \(R = AAASAAASAAA\), then \(c(v_1) \cdots c(v_{13}) = 1231213212312\) and
$v_3v_2v_1v_2v_3v_4v_5v_6v_7v_8v_9v_{10}v_{11}v_{12}v_{13}v_{12}$
is a stroll that induces the repetitive sequence $3212312132123121$.
\end{proof}

Now, we prove that there is no \(\mathcal{H}\)-free \(SA\)-sequence of order at least \(20\).

\begin{lemma}\label{lemma:longest-H-free-words}
The longest \(\mathcal{H}\)-free word has length \(19\).
\end{lemma}

\begin{proof}
Suppose \(R\) is a word on the alphabet \(\{S,A\}\).
If \(R\) does not contain \(SS\) nor \(AAAA\) as subwords,
then \(R\) is of the form \(R_1SR_2S \cdots R_{\ell-1}SR_\ell\)
with \(R_1,R_\ell\in\{\emptyset,A,AA,AAA\}\) and \(R_i\in\{A,AA,AAA\}\) for \(i\in\{2,\ldots,\ell-1\}\).
Suppose that \(R_i = A\) for \(2\leq i\leq \ell-1\).
If \(R_{i-1}\neq \emptyset\) and \(R_{i+1}\neq \emptyset\), i.e., \(R_{i-1},R_{i+1}\in\{A,AA,AAA\}\),
then \(R_{i-1}SR_iSR_{i+1}\) contains the subword \(ASASA\).
Therefore, that either \(R_{i-1} = \emptyset\) or \(R_{i+1} = \emptyset\).
In other words, if \(R_i = A\), then \(i\in\{2,\ell-1\}\),
and hence \(R_3,\ldots,R_{\ell-2}\in\{AA,AAA\}\).
Analogously, if \(R_i = AA\) with \(i\in\{3,\ell-2\}\), then either \(R_{i-1} = A\) or \(R_{i+1} = A\),
otherwise we find \(AASAASAA\) as subword.
This implies that \(R_4 =\cdots =  R_{\ell-3} = AAA\),
but there is no~\(i\) for which \(R_{i-1} = R_i = R_{i+1} = AAA\),
otherwise we find \(AAASAAASAAA\) as subword.
Therefore, the longest \(\mathcal{H}\)-free word is \(SASAASAAASAAASAASAS\),
which has length~19.
\end{proof}

As a consequence of Lemmas~\ref{3-colStrollNonRepetitiveBlocks} and \ref{lemma:longest-H-free-words}, 
and K{\"u}ndgen and Pelsmajer's~\cite{kundgen2008nonrepetitive} result that \(\sigma(P_n)\leq 4\), we obtain the desired result.
\begin{corollary}\label{cor:long-paths-snc}
If \(n\geq 22\), then \(\rho(P_n) = 4\).
\end{corollary}

\begin{proof}
Let \(P_n\) be a path with \(n\geq 22\).
Since \(\rho(G) \leq \sigma(G)\) for every graph \(G\), we have \(\rho(P_n)\leq \sigma(P_n) \leq 4\).
Suppose \(\rho(P_n) \leq 3\), and let \(c\) be a stroll-nonrepetitive \(3\)-coloring of \(P_n\).
Since \(n\geq 22\), the \(SA\)-sequence of \(P_n\)
has length at least \(20\), and by Lemma~\ref{lemma:longest-H-free-words},
contains a word in \(\mathcal{H}\), a contradiction to Lemma~\ref{3-colStrollNonRepetitiveBlocks}.
\end{proof}

Theorem~\ref{thm:snrc-paths} then follows from Proposition~\ref{prop:small-paths-snc} and Corollary~\ref{cor:long-paths-snc}.
In what follows, we prove Theorem~\ref{thm:snrc-cycles},
which settles the stroll-nonrepetitive chromatic number of cycles.

\snrcycles*

\begin{proof}
    Let \(C_n = v_1 \cdots v_n v_1\).
    First, note that any proper \(3\)-coloring without empty colors of \(C_3\) and of \(C_4\) has two colors that appears only once,
    and hence is a stroll-nonrepetitive coloring.
    Moreover no \(2\)-coloring of \(C_4\) is stroll-nonrepetitive.
    Therefore, \(\rho(C_3) = \rho(C_4) = 3\).
    Thus, we may assume \(n\geq 5\).
    Since \(\rho(C_5) \geq \pi(C_5)\) and \(\pi(C_5) = 4\) by Theorem~\ref{thm:currie-path-nonrep-cycle}, we have \(\rho(C_5)\geq 4\).
    Now, note that we can find a proper \(4\)-coloring of \(C_5\)
    in which only one color appears more than once.
    Then it has no repetitive stroll.
    
    Now, to prove that \(\rho(C_6) = 3\), consider the coloring given by \(c(v_1)\cdots c(v_6) = 212313\).
    Since one occurrence of \(1\) is adjacent only to vertices with color \(2\)
    and the other occurrence of \(1\) is adjacent only to vertices with color \(3\),
    no repetitive stroll contains a vertex with color \(1\).
    But, as seen before, paths with three vertices do not contain strolls.
    
    To prove that \(\rho(C_7) \leq 4\), color \(v_1,\ldots, v_6\) with 
    a stroll-nonrepetitive \(3\)-coloring of \(P_6\)
    and color \(v_7\) with color \(4\).
    Since \(4\) appears only once, no repetitive stroll contains \(v_7\).
    Then every stroll is contained in \(v_1,\ldots, v_6\), 
    and hence is nonrepetitive.
    Now, analogously to the case of \(C_5\),
    by Theorem~\ref{thm:currie-path-nonrep-cycle}, we have \(\rho(C_7)\geq 4\),
    which concludes this case.
    
    To prove that \(\rho(C_8) = 3\), consider the coloring given by \(c(v_1) \cdots c(v_8) = 12132123\).
    Let \(w = w_1 \cdots w_{2t}\) be a stroll in \(C_8\). 
    If \(w\) does not contain \(v_8\), then \(w\) cannot contain \(v_4\) which is the only other vertex with color \(3\), and hence \(w\) is in a path with three vertices, a contradiction.
    Therefore, we may assume that \(w\) contains \(v_4\) and \(v_8\) and \(t\geq 4\).
    Suppose that \(w_i = v_4\) and \(w_{i+t} = v_8\).
    Moreover, we assume that \(i\) is the smallest index with such that that \(w_i = v_4\) and \(w_{i+t} = v_8\).
    Suppose that \(i\geq 3\).
    Suppose, moreover, that \(w_{i-1} = v_3\).
    Then, by the minimality of \(i\), \(w_{i-2} = v_2\).
    This implies that \(c(w_{i-2})c(w_{i-1})c(w_i) = c(w_{i+t-2})c(w_{i+t-1})c(w_{i+t}) = 213\),
    but then \(w_{i+t-2} = v_2\), and \(w\) is not a stroll.
    Analogously, if \(w_{i-1} = v_5\), we get \(w_{i-2} = w_{i+t-2} = v_6\),
    again a contradiction.
    Thus, we may assume \(i\leq 2\), and hence \(i+2 \leq 4 \leq t\).
    In this case we use a similar argument with \(c(w_{i})c(w_{i+1})c(w_{i+2})\):
    if \(w_{i+1} = v_3\) then \(w_{i+2} = v_2 = w_{i+t+2}\),
    and if \(w_{i+1} = v_5\), then \(w_{i+2} = v_6 = w_{i+t+2}\).
    In both cases we reach a contradiction.

    Now, suppose \(n\geq 9\).
    By Theorem~\ref{thm:wnrc-cycles}, we have \(\sigma(C_n) \leq 4\).
    Suppose, for a contradiction, that \(C_n\) has a stroll-nonrepetitive \(3\)-coloring \(c\).
    We use Lemma~\ref{3-colStrollNonRepetitiveBlocks} on subpaths of \(C_n\).
    Let \(R = a_1\cdots a_n\) be the \(SA\)-sequence of \(c\).
    By Lemma~\ref{3-colStrollNonRepetitiveBlocks} the \(SA\)-sequence of any subpath of \(C_n\) 
    is \(\mathcal{H}\)-free.
    Then since \(C_n\) is a cycle, we may assume that \(R\) is of the form \(R_1SR_2S \cdots R_\ell S\)
    with \(R_i\in\{A,AA,AAA\}\) for \(i\in\{1,\ldots,\ell\}\).
    In what follows,  sums and subtractions in the indexes are taken modulo \(\ell\),
    i.e., \(R_\ell = R_0\) and \(R_{\ell+1} = R_1\).
    Since \(n\geq 9\) we have \(\ell\geq 3\).
    Now, for \(i\in\{1,\ldots,\ell\}\) we have \(R_i\neq A\) 
    otherwise \(R_{i-1}SR_iSR_{i+1}\) contains \(ASASA\) as a subword corresponding to a subpath of \(C_n\) of order \(7\), a contradiction to Lemma~\ref{3-colStrollNonRepetitiveBlocks}.
    Now, if \(R_i = AA\), then \(R_{i-1}SR_iSR_{i+1}\) contains \(AASAASAA\),
    which does not correspond to a subpath of \(C_n\) only when \(n = 9\).
    But if \(n = 9\), we have \(c(v_1) \cdots c(v_9) = 123132123\) and \(c(v_8)c(v_9)c(v_1)c(v_2)c(v_3)c(v_4) = 231231\) is a repetitive path.
    Therefore, we can assume that \(R_i = AAA\) for every \(i\in\{1,\ldots,\ell\}\).
    In particular, this implies that \(n \equiv 0\pmod{4}\) and, since \(\ell\geq 3\) we have \(n\geq 12\).
    Note that if \(n\geq 16\), then \(R\) contains \(AAASAAASAAA\) as a subword, a contradiction to Lemma~\ref{3-colStrollNonRepetitiveBlocks}.
    Therefore, we may assume \(n = 12\).
    Suppose, without loss of generality, that \(a_2 \cdots a_{11} = AAASAAASAA\).
    Thus, assuming \(c(v_1)c(v_2) = 12\), we have \(c(v_1) \cdots c(v_{12}) = 123121321231\),
    and hence\(c(v_{12})c(v_1) = 11\), a contradiction. 
    This completes the proof.
\end{proof}

\section{Concluding remarks}

In this paper we settled the walk- and stroll-nonrepetitive chromatic numbers of paths and cycles.
We hope that the strategies used here can be extended to other classes of graphs
as graphs with maximum degree \(3\) or general trees.
The chromatic numbers found here may be also useful for improving other variations of chromatic number on graphs 
consisting of a circular chain of isomorphic blocks,
as, for example, circular ladders~\cite{botler2023coloraccao}, powers of cycles, and special snarks.

\bibliographystyle{abbrv}\small
\bibliography{localbibliography}

\end{document}